\documentclass[11pt]{amsart}  
\usepackage{amsbsy,amsmath,amsthm,amssymb,color,verbatim,hyperref,mathtools}
\usepackage{fullpage}
\usepackage{float}
\usepackage[dvipsnames]{xcolor}
\usepackage{graphicx}
\usepackage{multicol}
\usepackage{multirow}
\usepackage{makecell}
\setcellgapes{4pt}
\usepackage[latin1]{inputenc}

\usepackage{tikz}
\usepackage{tikz-cd} 
\usepackage{tkz-graph}
\usepackage{xcolor}
\tikzstyle{vertex}=[circle, draw, inner sep=0pt, minimum size=4pt]

\tikzstyle{vtx}=[circle, draw, inner sep=0pt, minimum size=8pt]

\definecolor{darkgreen}{cmyk}{.9,0,.9,.2}
\definecolor{midgray}{gray}{0.60}
\definecolor{lightgray}{gray}{0.90}
\definecolor{lmgray}{gray}{0.70}
\usepackage{colortbl}

\def\a{\alpha}

\def\NN{{\mathbb N}}
\allowdisplaybreaks

\def\a{\alpha}
\def\w{\varpi}

\def\sgn{\text{sgn}}

\def\P{a_0\wedge b_0}
\def\notP{a_1\vee b_1}
\def\Q{c_0\wedge b_0}
\def\notQ{c_1\vee b_1}
\def\R{a_0\wedge d_0}
\def\notR{a_1\vee d_1}
\def\S{c_0\wedge e_0}
\def\notS{c_1\vee e_1}
\def\T{f_0\wedge d_0}
\def\notT{f_1\vee d_1}
\newcommand{\state}[5]{(#1)\wedge (#2)\wedge(#3)\wedge (#4)\wedge (#5)}

\newcommand{\RN}[1]
    {\MakeUppercase{\romannumeral #1}}

\def\sp{\mathfrak{sp}}
\usepackage{relsize}
\makeatletter
\newtheorem*{rep@theorem}{\rep@title}
\newcommand{\newreptheorem}[2]{%
\newenvironment{rep#1}[1]{%
 \def\rep@title{#2 \ref{##1}}%
 \begin{rep@theorem}}%
 {\end{rep@theorem}}}
\makeatother
\makeatletter
\newtheorem*{rep@proposition}{\rep@title}
\newcommand{\newrepproposition}[2]{%
\newenvironment{rep#1}[1]{%
 \def\rep@title{#2 \ref{##1}}%
 \begin{rep@proposition}}%
 {\end{rep@proposition}}}
\makeatother
\makeatletter
\newcommand{\addresseshere}{%
  \enddoc@text\let\enddoc@text\relax
}
\makeatother

\theoremstyle{definition}
\newtheorem{definition}{Definition}
\newtheorem{example}{Example}

\newtheorem{theorem}{Theorem}[section]
\newreptheorem{theorem}{Theorem}
\newrepproposition{proposition}{Proposition}
\newtheorem{proposition}{Proposition}[section] 

\newtheorem{lemma}{Lemma}[section] 
\newtheorem{corollary}{Corollary}[section]

\newcommand{\floor}[1]{\left\lfloor #1 \right\rfloor}

\title{Weight $q$-multiplicities for representations of the exceptional Lie algebra $\mathfrak{g}_2$}

\author{Jerrell Cockerham}
\address[Jerrell Cockerham]{Department of Mathematics and Computer Science, 
Colorado College, United States}\email{\textcolor{blue}{\href{mailto:jerrell.cockerham@gmail.com }{jerrell.cockerham@gmail.com }}}

\author{Melissa Guti\'errez Gonz\'alez}
\address[Melissa Guti\'errez Gonz\'alez]{Department of Mathematics,
Occidental College, United States}
\email{\textcolor{blue}{\href{mailto:mgutierrezgo@oxy.edu}{mgutierrezgo@oxy.edu}}}

\author{Pamela E. Harris}
\address[Pamela E. Harris]{Department of Mathematics and Statistics,
Williams College, United States}
\email{\textcolor{blue}{\href{mailto:peh2@williams.edu}{peh2@williams.edu}}}

\author{Marissa Loving}
\address[Marissa Loving]{School of Mathematics, Georgia Tech, United States}
\email{\textcolor{blue}{\href{mailto:mloving6@gatech.edu}{mloving6@gatech.edu}}}

\author{Amaury V. Mini\~no}
\address[Amaury V. Mini\~no]{Department of Mathematical Sciences,
Florida Atlantic University, United States}
\email{\textcolor{blue}{\href{mailto:aminino2017@fau.edu}{aminino2017@fau.edu}}}

\author{Joseph Rennie}
\address[Joseph Rennie]{Department of Mathematics, University of Illinois at Urbana-Champaign, United States}
\email{\textcolor{blue}{\href{mailto:rennie2@illinois.edu}{rennie2@illinois.edu}}}

\author{Gordon Rojas Kirby}
\address[Gordon Rojas Kirby]{Department of Mathematics, University of California, Santa Barbara, United States}
\email{\textcolor{blue}{\href{mailto:gkirby@math.ucsb.edu}{gkirby@math.ucsb.edu}}}

\subjclass[2010]{17B10}
\keywords{$q$-analog of Kostant's partition function; $q$-weight multiplicities; exceptional Lie algebra~$\mathfrak{g}_2$}

\date{\today}
\begin{document}

\maketitle

\begin{abstract}
Given a simple Lie algebra $\mathfrak{g}$, Kostant's weight $q$-multiplicity formula is an alternating sum over the Weyl group whose terms involve the $q$-analog of Kostant's partition function. For $\xi$ (a weight of $\mathfrak{g}$), the $q$-analog of Kostant's partition function is a polynomial-valued function defined
by $\wp_q(\xi)=\sum c_i q^i$ where $c_i$ is the number of ways $\xi$ can be written as a sum of $i$ positive roots of $\mathfrak{g}$.  
In this way, the evaluation of Kostant's weight $q$-multiplicity formula at $q = 1$ recovers the multiplicity of a weight in a highest weight representation of $\mathfrak{g}$.
In this paper, we give closed formulas for computing weight $q$-multiplicities in a highest weight representation of the exceptional Lie algebra  $\mathfrak{g}_2$. 
\end{abstract}

\section{Introduction}
We recall that the theorem of the highest weight asserts that a finite-dimensional complex irreducible representation of 
a simple Lie algebra $\mathfrak{g}$ is  equivalent to $L(\lambda)$, a highest weight representation with dominant integral highest weight $\lambda$.
The multiplicity of a weight $\mu$ in $L(\lambda)$, denoted by $m(\lambda,\mu)$, can be computed using Kostant's weight multiplicity formula (as defined by Kostant in \cite{Kostant}):
\begin{equation}
m(\lambda,\mu) = \sum_{\sigma\in W}^{} (-1)^{\ell(\sigma)}\wp(\sigma(\lambda + \rho) - (\mu + \rho))\label{eq:KWMF}
\end{equation} 
where $W$ is the Weyl group of $\mathfrak{g}$, $\ell(\sigma)$ denotes the length of $\sigma\in W$, and $\rho = \frac{1}{2}\sum_{\alpha\in\Phi^+}^{} \alpha$ with $\Phi^+$ being the set of positive roots of $\mathfrak{g}$, and where
$\wp$ denotes Kostant's partition function, which counts the number of ways to express a weight as a nonnegative integral sum of positive roots.

In this paper, we consider the exceptional Lie algebra $\mathfrak{g}_2$
and study the $q$-analog of Kostant's weight multiplicity formula, also known as Kostant's weight $q$-multiplicity formula,  which is a generalization of equation \eqref{eq:KWMF} defined by Luztig in \cite{Lusztig}:
\begin{equation}
m_q(\lambda,\mu) = \sum_{\sigma\in W}^{} (-1)^{\ell(\sigma)}\wp_q(\sigma(\lambda + \rho) - (\mu + \rho)).\label{eq:qKWMF}
\end{equation}
In equation \eqref{eq:qKWMF},   $\wp_q$ denotes the $q$-analog of Kostant's partition function, which is a polynomial-valued function defined by
\begin{align}\wp_q(\xi)&=c_0+c_1q + c_2q^2 + \cdots+ c_nq^n,\label{eq:qKPF} \end{align}
where $c_i$ denotes the number of ways to express the weight $\xi$ as a sum of exactly $i$ positive roots. Note that equation \eqref{eq:qKWMF} generalizes \eqref{eq:KWMF} since $\wp_q(\xi)|_{q=1}=\wp(\xi)$ for any weight $\xi$ and so  $m_q(\lambda,\mu)|_{q=1}=m(\lambda,\mu)$. One important application of equation \eqref{eq:qKWMF} is the celebrated result of Lusztig ~\cite[Section 10, p.~226]{Lusztig}, which states that if $\mathfrak{g}$ is a finite-dimensional simple Lie algebra $\mathfrak{g}$ and $\tilde{\alpha}$ is the highest root, then $m_q(\tilde{\alpha},0)=q^{e_1}+q^{e_2}+\cdots+q^{e_r}$ where $e_1,e_2,\ldots,e_r$ are the exponents of $\mathfrak{g}$.  In the case of the exceptional Lie algebra  $\mathfrak{g}_2$, this implies that $m_q(\tilde{\alpha},0)=q+q^5$.

Although formulas such as equation \eqref{eq:KWMF} and \eqref{eq:qKWMF} exist, it is very difficult to give closed formulas for weight multiplicities for a Lie algebra  of arbitrary rank. The difficulties in this work arise from both the lack of closed formulas for the partition functions involved, as well as the factorial growth of the Weyl group order as the rank of the the Lie algebra increases. For some results related to computations of weight multiplicities in certain highest weight representations see \cite{CHI2020,H2011,PH2017,HIS2018,HIW2016,HRSS2019}. In general, there has been some success in providing closed formulas for weight $q$-multiplicities for Lie algebras of low rank. This includes the work of Harris and Lauber \cite{harris2017weight} on  weight $q$-multiplicities for the representations of $\mathfrak{sp}_{4}(\mathbb{C})$, which generalized the the work of Refaghat and Shahryari \cite{RS}, and the work of Garcia, Harris, Loving, Martinez, Melendez, Rennie, Rojas Kirby, and Tinoco \cite{LucyDavid2020} on  weight $q$-multiplicities for $\mathfrak{sl}_4(\mathbb{C})$. Other work provides visualizations of the subsets of elements of the Weyl group which contribute non-trivially to the associated weight multiplicity, for examples see \cite{HLM2017,rank2-2019}. Motivated by these works, we present a new formula for equation \eqref{eq:qKWMF} giving weight $q$-multiplicities for representations of the exceptional Lie algebra $\mathfrak{g}_2$.
\begin{theorem}\label{thm:main}
Let $\varpi_1$ and $\varpi_2$ denote the fundamental weights of $\mathfrak{g}_2$. If $\lambda = m\varpi_1 + n\varpi_2$, $\mu = x\varpi_1 + y\varpi_2$, and $m, n, x, y \in \NN:=\{0,1,2,3,\ldots\}$, then 

\begin{equation}\label{multiplicity formula piecewise}
      m_{q}(\lambda,\mu)=\begin{cases} 
      P-Q-R+S+T & \text{if and only if}  \; \; a,b,c,d,e,f\in \NN, \\
      P-Q-R+S & \text{if and only if} \; \; a,b,c,d,e\in \NN, f \not \in \NN, \\
      P-Q-R+T & \text{if and only if} \; \; a,b,c,d,f\in \NN, e\not \in \NN, \\
      P-Q-R & \text{if and only if} \; \; a,b,c,d\in \NN, e,f\not \in \NN, \\
      P-Q & \text{if and only if}  \; \; a,b,c\in \NN, d,e,f\not \in \NN, \\
      P-R & \text{if and only if}  \; \; a,b,d\in \NN, c,e,f \not \in \NN, \\
      P & \text{if and only if} \; \; a,b\in \NN, c,d,e,f\not \in \NN, \\
      0 & \text{otherwise} \\
  \end{cases}
   \end{equation}
where
\begin{align}\label{PQRST evaluations}
    \begin{split}
     P &= 
        \wp_q((2m +3n -2x -3y)\alpha_1 + (m +2n -x -2y)\alpha_2),
    \\
    Q &= \wp_q((m +3n -2x -3y -1)\alpha_1 + (m +2n -x -2y)\alpha_2),
     \\
    R &= 
    \wp_q((2m +3n -2x -3y)\alpha_1 + (m +n -x -2y -1)\alpha_2),
    \\
     S &= 
     \wp_q((m +3n -2x -3y -1)\alpha_1 + {(n-x-2y-2)}\alpha_2),\mbox{ and}
     \\
     T &= \wp_q({(m-2x-3y -4)}\alpha_1 + {(m +n -x -2y -1)}\alpha_2).
     \end{split}
\end{align}
\end{theorem}
In general, using equation \eqref{eq:qKWMF} to compute weight $q$-multiplicities for representations of $\mathfrak{g}_2$ requires the computation of Kostant's partition function on 12 distinct inputs, as the Weyl group of $\mathfrak{g}_2$ is isomorphic to the dihedral group of order 12. However, Theorem~\ref{thm:main} reduces all weight $q$-multiplicity computations to at most five such computations. Our second result, provides a formula for the $q$-analog of Kostant's partition function for $\mathfrak{g}_2$, which can be used to compute each of the terms appearing in Theorem \ref{thm:main}.
\begin{proposition}\label{prop:G2qpartition}
If $m,n \in \NN$, then the value of $\wp_q(m\alpha_1 + n\alpha_2)$ is given by 
\begin{align}
        \sum_{i=0}^{\min \left( \lfloor \frac{m}{3} \rfloor , \floor{\frac{n}{2}} \right)} 
        \left( \sum_{j=0}^{\min \left( \floor{\frac{m-3i}{3}} , n-2i \right)} 
        \left( \sum_{k=0}^{\min \left( \lfloor \frac{m-3i-3j}{2} \rfloor , n-2i-j \right)} 
        \left( \sum_{l=0}^{\min \left( m-3i-3j-2k , n-2i-j-k \right)} q^{z}  \right) \right) \right),
\end{align}
where $z=m+n-4i-3j-2k-l$.
\end{proposition}
\textbf{Outline of the paper.}
Section \ref{sec:background} provides the Lie theoretic background needed for the remainder of the manuscript. Section \ref{sec:qpartition} contains the proof of Proposition \ref{prop:G2qpartition}. We prove Theorem \ref{thm:main} in Section~\ref{qKWMF} and provide some detailed examples of how Theorem \ref{thm:main} can be used to compute weight $q$-multiplicities for representations of $\mathfrak{g}_2$. In Section \ref{sec:correction}, we provide a missing case in the proof of a formula of Harris and Lauber for the $q$-analog of Kostant's partition function of the Lie algebra $\mathfrak{sp}_4(\mathbb{C})$ appearing in~\cite{harris2017weight}. We end the manuscript with a section containing some open problems.

\section{Background}\label{sec:background}
We use the same notation as appearing in \cite{GW}, which the reader can look to for a more comprehensive treatment of some of the objects introduced here. We denote the simple roots of $\mathfrak{g}_2$ as $\alpha_1$ and $\alpha_2$, and the fundamental weights as $\varpi_1$ and $\varpi_2$. The positive roots of $\mathfrak{g}_2$ are given by \[\Phi^{+} = \{\alpha_1, \alpha_2,\alpha_1 + \alpha_2, 2\alpha_1 + \alpha_2, 3\alpha_1 + \alpha_2, 3\alpha_1 + 2\alpha_2\}.\] Recall that $\varpi_1 = 2 \alpha_1 +  \alpha_2$, $\varpi_2 = 3 \alpha_1 + 2 \alpha_2,$ and 
\begin{equation}\label{rho}
    \rho = \frac{1}{2} \sum_{\alpha \in \Phi^{+}} \alpha=\varpi_1+\varpi_2=5\alpha_1+3\alpha_2.
\end{equation} 
We set $\lambda = (2m+3n)\alpha_1 + (m+2n)\alpha_2$ and $\mu = (2x+3y)\alpha_1 + (x+2y)\alpha_2$ where $m, n, x, y \in \NN$. We make this choice to simplify our computations and we are able to do so since the fundamental weight lattice and the root lattice of $\mathfrak{g}_2$ are equal.

The Weyl group of $\mathfrak{g_2}$, denoted $W$, is generated by reflections about hyperplanes orthogonal to the simple roots. We denote the reflection through the hyperplane orthogonal to $\alpha_i$ by $s_i$ for $i = 1, 2$. In Figure \ref{root system}, we illustrate the positive roots and in red we present the hyperplanes defining the reflections $s_1$ and $s_2$. The action of the generators of $W$ on the simple roots is given by
\begin{align}
    s_1(\alpha_1) & = -\alpha_1, & & s_1(\alpha_2) = 3\alpha_1 + \alpha_2,\label{alpha1} \\
    s_2(\alpha_1) & = \alpha_1 + \alpha_2, & & s_2(\alpha_2) = -\alpha_2 \label{alpha2}.
\end{align} 
Table \ref{elements of weyl group} describes how the remaining elements of $W$ act on the simple roots. 
\begin{figure}[htb!]
\centering
    \begin{tikzpicture}
        \draw[red, very thick] (90:-3) to (90:3.5);
        \draw[red, very thick] (60:-3) to (60:3.5);
        \foreach\ang in {60,120,...,360}{
         \draw[->,blue!80!black,thick] (0,0) -- (\ang:1.5cm);
        }
        \foreach\ang in {30,90,...,330}{
         \draw[->,blue!80!black,thick] (0,0) -- (\ang:2.5cm);
        }
        \node[anchor=north west,scale=1] at (-2.8, 1.6) {$\alpha_{2}$};
        \node[anchor=south west,scale=1] at (1.5,-.25) {$\alpha_{1}$};
        \node[anchor=north,scale=1] at (-1, 2.6) {$3\alpha_1 + 2\alpha_2$};
        \node[anchor=north east,scale=1] at (1.8, 1.8) {$2\alpha_1 + \alpha_2$};
        \node[anchor=north west,scale=1] at (-1.7, 1.8) {$\alpha_1 + \alpha_2$};
        \node[anchor=north east,scale=1] at (3.9, 1.6) {$3\alpha_1 + \alpha_2$};
        \node[anchor=north,scale=1] at (0, 4) {$s_1$};
        \node[anchor=north east,scale=1] at (2.2, 3.5) {$s_2$};
    \end{tikzpicture}
    \caption{Positive root system for $\mathfrak{g_2}$ and the lines orthogonal to the simple roots which define $s_1$ and $s_2$.}\label{root system}
\end{figure}
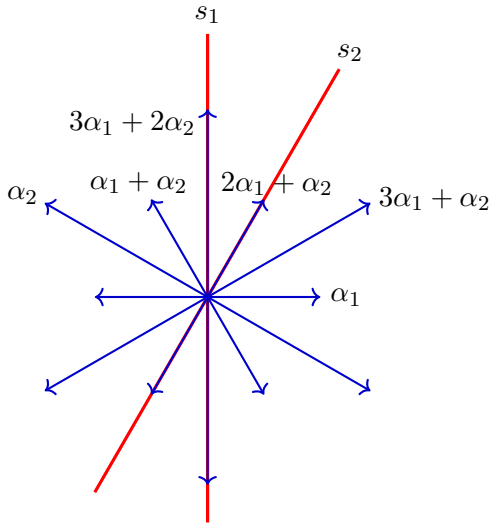

\setlength\extrarowheight{1pt}
\begin{table}[htb!]
    \centering
    \begin{tabular}{| c || c | c | c | c | c | c |}
        \hline
        $\sigma \in W$ & $1$ & $s_1$ & $s_2s_1$ & $s_1s_2s_1$ & $(s_2s_1)^2$ & $s_1(s_2s_1)^2$ \\
        \hline
        $\sigma(\alpha_1)$ & $\alpha_1$ & $-\alpha_1$ & $-(\alpha_1 + \alpha_2)$ & $-(2\alpha_1 + \alpha_2)$ & $-(2\alpha_1 + \alpha_2)$ & $-(\alpha_1+\alpha_2)$\\
        
        \hline
       $\sigma(\alpha_2)$ & $\alpha_2$ & $3\alpha_1 + \alpha_2$ & $3\alpha_1 + 2\alpha_2$ & $3\alpha_1 + 2\alpha_2$ & $3\alpha_1 + \alpha_2$ & $\alpha_2$ \\
        \hline 
        \hline 
        $\sigma \in W$ & $s_2$ & $s_1s_2$ & $s_2s_1 s_2$ & $(s_1s_2)^2$ & $s_2 (s_1 s_2)^2$ & $(s_1 s_2)^3$\\
        \hline
        $\sigma(\alpha_1)$ & $\alpha_1 + \alpha_2$ & $2\alpha_1 + \alpha_2$ & $2\alpha_1 + \alpha_2$ & $\alpha_1 + \alpha_2$ & $\alpha_1$ & $-\alpha_1$ \\
        \hline
        $\sigma(\alpha_2)$ & $-\alpha_2$ & $-(3\alpha_1 + \alpha_2)$ & $-(3\alpha_1 + 2\alpha_2)$ & $-(3\alpha_1 + 2\alpha_2)$ & $-(3\alpha_1 + \alpha_2)$ & $-\alpha_2$ \\
        \hline
    \end{tabular}
    \caption{Elements of $W$ and their action on the simple roots $\alpha_1$ and $\alpha_2$.}
    \label{elements of weyl group}
\end{table}

\section{The \texorpdfstring{$q$}{q}-analog of Kostant's Partition Function}
\label{sec:qpartition}
\setlength\extrarowheight{0pt}

In this section, we provide a closed formula for the $q$-analog of Kostant's partition function for the exceptional Lie algebra $\mathfrak{g}_2$, which was presented in equation \eqref{eq:qKPF}. We restate the result below for ease of reference.
\begin{repproposition}{prop:G2qpartition}
\emph{If $m,n \in \NN$, then the value of $\wp_q(m\alpha_1 + n\alpha_2)$ is given by }
\begin{align}
        \sum_{i=0}^{\min \left( \lfloor \frac{m}{3} \rfloor , \floor{\frac{n}{2}} \right)} 
        \left( \sum_{j=0}^{\min \left( \floor{\frac{m-3i}{3}} , n-2i \right)} 
        \left( \sum_{k=0}^{\min \left( \lfloor \frac{m-3i-3j}{2} \rfloor , n-2i-j \right)} 
        \left( \sum_{l=0}^{\min \left( m-3i-3j-2k , n-2i-j-k \right)} q^{z}  \right) \right) \right),
\end{align}
\emph{where} $z=m+n-4i-3j-2k-l$.
\end{repproposition}
\begin{proof}
The number of ways we can write $m\alpha_1 + n\alpha_2$ as a nonnegative integral sum of positive roots is determined by the number of times each positive root in \[\Phi^{+} = \{\alpha_1, \alpha_2,\alpha_1 + \alpha_2, 2\alpha_1 + \alpha_2, 3\alpha_1 + \alpha_2, 3\alpha_1 + 2\alpha_2\}\] is used. 

If a partition includes $i$ multiples of the highest root $3\alpha_1 + 2\alpha_2$, then $0 \leq i \leq \min(\lfloor \frac{m}{3} \rfloor, \lfloor \frac{n}{2} \rfloor) $, so as to not exceed each coefficient of the weight $m\alpha_1 + n\alpha_2$ for $m$ and $n$. 
We are now left to partition $m\alpha_1 + n\alpha_2 - i(3\alpha_1 + 2\alpha_2) = (m - 3i)\alpha_1 + (n - 2i)\alpha_2$. 
If the partition of  $(m - 3i)\alpha_1 + (n - 2i)\alpha_2$ includes $j$ multiples of the root $3\alpha_1 + \alpha_2$, then $0 \leq j \leq \min( \left \lfloor \frac{m-3i}{3} \right \rfloor, n-2i)$. 
In which case, we must partition $(m - 3i)\alpha_1 + (n - 2i)\alpha_2 - j(3\alpha_1 + \alpha_1) = (m - 3i - 3j)\alpha_1 + (n - 2i - j)\alpha_2$. If the partition of $(m - 3i - 3j)\alpha_1 + (n - 2i - j)\alpha_2$ includes $k$ multiples of the root $2\alpha_1 + \alpha_1$, then $0 \leq k \leq \min( \lfloor \frac{m-3i-3j}{2} \rfloor, n-2i-j)$. We must now partition
$(m - 3i - 3j)\alpha_1 + (n - 2i - j)\alpha_2 - k(2\alpha_1 + \alpha_2) = (m - 3i - 3j - 2k)\alpha_1 + (n - 2i - j - k)\alpha_2$. 
If the partition of $(m - 3i - 3j - 2k)\alpha_1 + (n - 2i - j - k)\alpha_2$ includes $l$ multiples of $\alpha_1 + \alpha_2$, then $0 \leq l \leq \min(m - 3i - 3j -2k, n - 2i - j - k)$. 
We are left to partition $(m - 3i - 3j - 2k)\alpha_1 + (n - 2i - j - k)\alpha_2 - l(\alpha_1 + \alpha_2) = (m - 3i - 3j - 2k - l)\alpha_1 + (n - 2i - j - k - l)\alpha_2$. Finally, the coefficients of $\alpha_1$ or $\alpha_2$ in our partition are determined by our choice of $i, j, k, l$ and are $m - 3i - 3j - 2k - l$ and $n - 2i - j - k - l$, respectively.

It follows that the total number of roots used is given by $z = i + j + k + l + (m - 3i - 3j - 2k - l) + (n - 2i - j - k - l) = m + n - 4i - 3j - 2k - l$.
\end{proof}
With the formula of Proposition \ref{prop:G2qpartition} at hand, next we compute the values of $\sigma(\lambda+\rho)-(\mu+\rho)$ as they appear in \eqref{eq:qKWMF} for each $\sigma \in W$. Recall that $\lambda = (2m+3n)\alpha_1 + (m+2n)\alpha_2$ and $\mu = (2x+3y)\alpha_1 + (x+2y)\alpha_2$, where $m, n, x, y \in \NN$. To illustrate the computations, we consider the case when $\sigma = s_1$, and using equations \eqref{rho}, \eqref{alpha1}, and \eqref{alpha2}, we find that
\begin{align*}
    s_1(\lambda +\rho) & -(\mu+\rho) \\
    &= s_1((2m+3n)\alpha_1 + (m+2n)\alpha_2 + 5\alpha_1 + 3\alpha_2) - ((2x+3y)\alpha_1 + (x+2y)\alpha_2 + 5\alpha_1 + 3\alpha_2) \\
    &= ((2m+3n+5)(-\alpha_1) + (m+2n+3)(3\alpha_1 + \alpha_2) - (2x +3y+5)\alpha_1 - (x+2y+3)\alpha_2 \\
    &= (m+3n-2x-3y-1)\alpha_1+(m+2n-x-2y)\alpha_2.
\end{align*}
Repeating this process with every remaining Weyl group element yields  the contents of Table \ref{partition function evaluations}.

\setlength\extrarowheight{2pt}
\begin{table}[htb!]
    \centering 
    \begin{tabular}{|c| c| c| c| c| c| c| }
    \hline
    $\sigma$ & $\ell(\sigma)$ & $\sigma(\lambda+\rho)-(\mu+\rho)$
    \\
    \hline

    $1$ & 0 & $ \left(2m+3n-2x-3y\right)\alpha_1+\left(m+2n-x-2y\right)\alpha_2$ \\\hline
    
    $s_1$ & 1 & $ \left(m+3n-2x-3y-1\right)\alpha_1+\left(m+2n-x-2y\right)\alpha_2$ \\\hline
    
    $s_2$ & 1 &  $\left(2m+3n-2x-3y\right)\alpha_1+\left(m +n-x-2y-1\right)\alpha_2$ \\\hline
    
    $s_2s_1$ & 2 & $\left(m+3n-2x-3y -1\right)\alpha_1+\left(n-x-2y-2\right)\alpha_2$ \\\hline
    
    $s_1s_2$ & 2 & $\left(m-2x-3y-4\right)\alpha_1+\left(m+n-x-2y-1\right)\alpha_2$ \\\hline
  
    $s_1s_2s_1$ & 3 & $\left(-m-2x-3y-6\right)\alpha_1+\left(n-x-2y-2\right)\alpha_2$ \\\hline
    
    $s_2s_1s_2$ & 3 & $\left(m-2x-3y-4\right)\alpha_1+\left(-n-x-2y-4\right)\alpha_2$ \\\hline
    
    $(s_1s_2)^2$ & 4 & $\left(-m-3n-2x-3y-9\right)\alpha_1+\left(-n-x-2y-4\right)\alpha_2$ \\\hline
    
    $(s_2s_1)^2$ & 4 & $\left(-m-2x-3y-6\right)\alpha_1+\left(-m-n-x-2y-5\right)\alpha_2$ \\\hline
    
    $s_1(s_2s_1)^2$ & 5 & $\left(-2m-3n-2x-3y-10\right)\alpha_1+\left(-m-n-x-2y-5\right)\alpha_2$ \\\hline
    
    $s_2(s_1s_2)^2$ & 5& $\left(-m-3n-2x-3y-9\right)\alpha_1+\left(-m-2n-x-2y-6\right)\alpha_2$ \\\hline
    
    $(s_1s_2)^3$ & 6 & $\left(-2m-3n-2x-3y-10\right)\alpha_1+\left(-m-2n-x-2y-6\right)\alpha_2$ \\\hline
\end{tabular}\caption{Evaluations of $\sigma(\lambda+\rho)-(\mu+\rho)$ for $\sigma \in W$.}\label{partition function evaluations}
\end{table}
\setlength\extrarowheight{0pt}

Observe that for $m, n, x, y \in \NN$, the $q$-analog of Kostant's partition function evaluates to zero if the coefficient of either $\alpha_1$ or $\alpha_2$ is negative. Thus, given the computations appearing in Table \ref{partition function evaluations}, we note that the only elements of the Weyl group that contribute to Kostant's weight $q$-multiplicity formula are $1, s_1, s_2, s_2s_1,$ and $s_1s_2$. The remaining elements of $W$ never contribute and, hence, we disregard them moving forward. With these observations, we are now ready to prove Theorem \ref{thm:main} by evaluating $m_q(\lambda, \mu)$ as appearing in \eqref{eq:qKWMF}.

\section{The \texorpdfstring{$q$}{q}-analog of Kostant's Weight Multiplicity Formula}
\label{qKWMF}
\subsection{Evaluation of \texorpdfstring{$m_q(\lambda,\mu)$}{mq}} In the previous section, we established that $1, s_1, s_2, s_2s_1,$ and $s_1s_2$ are the only Weyl group elements that contribute nontrivially to $m_q(\lambda,\mu)$ whenever $\lambda={m\w_1+n\w_2}=(2m+3n)\a_1+(m+2n)\a_2$ and ${\mu={x\w_1+y\w_2}=(2x+3y)\a_1+(x+2y)\a_2}$ with $m,n,x,y\in \mathbb{N}$. For the sake of simplicity, we make the following change of variables
\begin{align}
\begin{split}
        a &= 2m +3n -2x -3y,
        \\
        b &= m +2n -x -2y,
        \\
        c &= m +3n -2x -3y -1,
        \\
        d &= m +n -x -2y -1,
        \\
        e &= n-x-2y-2, \text{and}
        \\
        f &= m-2x-3y -4.
        \end{split}\label{ABCDEF labels}
\end{align}
Utilizing this change of variables together with the evaluations in Table \ref{partition function evaluations} for $\sigma=1, s_1, s_2, s_2s_1,$ and $s_1s_2$, we obtain
\begin{align}
\begin{split}
        P&=\wp_q(1(\lambda + \rho) - (\mu + \rho) )=\wp_q( a\alpha_1+b\alpha_2),
        \\
        Q&=\wp_q(s_1(\lambda + \rho) - (\mu + \rho) )=\wp_q( c\alpha_1+b\alpha_2),
        \\
        R&=\wp_q(s_2(\lambda + \rho) - (\mu + \rho) )=\wp_q( a\alpha_1+d\alpha_2),
        \\
        S&=\wp_q(s_2s_1(\lambda + \rho) - (\mu + \rho) )= \wp_q({c}\alpha_1+{e}\alpha_2),\; \text{and}
        \\
        T&=\wp_q(s_1s_2(\lambda + \rho) - (\mu + \rho) )=\wp_q( {f}\alpha_1+{d}\alpha_2).
        \end{split}\label{KWMF Relabeling}
\end{align}
The expressions in equation \eqref{KWMF Relabeling} are precisely the expressions described in \eqref{PQRST evaluations} and are the terms needed to evaluate $m_q(\lambda,\mu)$. However, there can be instances where certain values of $m,n,x,y \in \NN$ result in some of the expressions in \eqref{KWMF Relabeling} being zero, while others remain nonzero. When an expression is zero we say it contributes trivially to the $q$-multiplicity; if instead the expression is nonzero, then we say it contributes nontrivially to the $q$-multiplicity.

From \eqref{KWMF Relabeling}, we know that there are at most five terms, namely $P,Q,R,S, \text{ and } T$ that can contribute to $m_q(\lambda,\mu)$ depending on the values of $m,n,x,y\in\mathbb{N}$. This gives us at most $2^5=32$ distinct possible formulas for $m_q(\lambda,\mu)$. In the work that follows, we will prove that of these $32$ distinct possible cases only $8$ can occur. 

As is standard, we let $\vee$ denote the Boolean operator {\it{or}}, and $\wedge$ denote the Boolean operator {\it{and}}. Note that $a,b,c,d,e,f$, as given in \eqref{ABCDEF labels}, are always integer quantities. Hence, when $a,b,c,d,e,f$ are nonnegative, then $P,Q,R,S$, and $T$ contribute nontrivially to $m_q(\lambda,\mu)$. To simplify notation, we define the statements 
\begin{center}
\begin{tabular}{cccccc}
$a_0 :\; a \geq 0$, & $a_1 :\; a < 0$, &
$b_0 :\; b \geq 0$, & $b_1 :\; b < 0$, &
     $c_0 :\; c \geq 0$, & $c_1 :\; c < 0$,\\
     $d_0 :\; d \geq 0$, & $d_1 :\; d < 0$, &
     $e_0 :\; e \geq 0$, & $e_1 :\; e < 0$, &
     $f_0 :\; f \geq 0$, & $f_1 :\; f < 0$.\\
\end{tabular}
\end{center}
Thus, by definition of Kostant's partition function we have that
\begin{align}\label{PQRSTtruelogicnotation}
    \begin{split}
    &P \; \; \text{contributes nontrivially if and only if} \; \;a_0 \wedge b_0 \; \; \text{holds true},
    \\
    &Q \; \; \text{contributes nontrivially if and only if} \; \; c_0 \wedge b_0 \; \; \text{holds true},
    \\
    &R \; \; \text{contributes nontrivially if and only if} \; \; a_0 \wedge d_0 \; \; \text{holds true},
    \\
   &S \; \; \text{contributes nontrivially if and only if} \; \; {c_0} \wedge {e_0} \; \; \text{holds true},
    \\
    &T \; \; \text{contributes nontrivially if and only if} \; \; {f_0} \wedge {d_0} \; \; \text{holds true}.
    \end{split}
\end{align}

Hence, 

\begin{align}\label{PQRSTfalselogicnotation}
\begin{split}
    &P \; \; \text{contributes trivially if and only if} \; \;a_1 \vee b_1 \; \; \text{holds true},
    \\
    &Q \; \; \text{contributes trivially if and only if} \; \; c_1 \vee b_1 \; \; \text{holds true},
    \\
    &R \; \; \text{contributes trivially if and only if} \; \; a_1 \vee d_1 \; \; \text{holds true},
    \\
   &S \; \; \text{contributes trivially if and only if} \; \; {c_1} \vee {e_1} \; \; \text{holds true},
    \\
    &T \; \; \text{contributes trivially if and only if} \; \; {f_1} \vee {d_1} \; \; \text{holds true}.
\end{split}
\end{align}

We briefly illustrate our method of proof via an example. From the descriptions in \eqref{PQRSTtruelogicnotation} and \eqref{PQRSTfalselogicnotation}, we know that $m_q(\lambda,\mu)=P-Q+{T}$ when $P,Q,{T}$ contribute nontrivially and $R,{S}$ contribute trivially. This implies that the following necessary condition must be true: 
\[(a_0\wedge b_0)\wedge(c_0\wedge b_0)\wedge(a_1\vee d_1)\wedge (c_1\vee e_1)\wedge(f_0\wedge d_0).\]
However, we note that such a logical statement contains $(a_0\wedge d_0)\wedge(a_1\vee d_1)$, which can never be true. This establishes that $m_q(\lambda,\mu)\neq P-Q+{T}$ whenever $m,n,x,y\in\mathbb{N}$. In this case, we would state that $P-Q+{T}$ is a \emph{forbidden $q$-multiplicity formula}. We now give a general definition.

\begin{definition}
Fix $\lambda=m\w_1+n\w_2$ and $\mu=x\w_1+y\w_2$ with $m,n,x,y\in\mathbb{N}$. Let $P,Q,R,S,T$ be as in \eqref{KWMF Relabeling}, with $\sgn(P)=\sgn(S)=\sgn(T)=1$ and $\sgn(Q)=\sgn(R)=-1$. For any subset $X\subseteq\{P,Q,R,S,T\}$, if $m_q(\lambda,\mu)\neq \sum_{x\in X}\sgn(x)x$, then  $\sum_{x\in X}\sgn(x)x$ is said to be a \emph{forbidden $q$-multiplicity formula}.
\end{definition}
Using this new definition along with the technique illustrated above we establish the following.  

\begin{lemma}\label{contradictions to 21 possible cases}
Let $\lambda=m\w_1+n\w_2$ and $\mu=x\w_1+y\w_2$ with $m,n,x,y\in\mathbb{N}$. If $P,Q,R,S,T$ are as in \eqref{KWMF Relabeling}, then the formulas $\sum_{x\in X}\sgn(x)x$, with $X\subseteq\{P,Q,R,S,T\}$, listed in Table \ref{PQRST simple contradictions listed} are forbidden $q$-multiplicity formulas.
\end{lemma}

\begin{table}[htb!]
\resizebox{\textwidth}{!}{
\begin{tabular}{|l | l | l |}
\hline
\textit{$m_q(\lambda,\mu)$} & \textit{Necessary Conditions} & \textit{Contradictions} \\
\hline
$P-Q+S+T$ & $\state{\P}{\Q}{\notR}{\S}{\T}  $ & $(a_0 \wedge d_0) \wedge(a_1 \vee d_1)$\\ 
\hline
$P-R+S+T$ & $\state{\P}{\notQ}{\R}{\S}{\T}  $ & $(c_0\wedge b_0) \wedge ( c_1 \vee b_1)$\\
\hline
$-Q-R+S+T$ & $  \state{\notP}{\Q}{\R}{\S}{\T} $  & $(a_0 \wedge b_0) \wedge (a_1 \vee b_1)$ \\ 
\hline
$P+S+T$ & $ \state{\P}{\notQ}{\notR}{\S}{\T}  $ & $(c_0\wedge b_0) \wedge ( c_1 \vee b_1)$\\  
\hline
 $P-R+S$ & $\state{\P}{\notQ}{\R}{\S}{\notT}$ & $(c_0 \wedge b_0) \wedge (c_1 \vee b_1) $
 \\
 \hline
 $P-Q+T$ &  $\state{\P}{\Q}{\notR}{\notS}{\T} $ & $(a_0 \wedge d_0) \wedge (a_1 \vee d_1)$
 \\
 \hline
  $-Q-R+S$ & $ \state{\notP}{\Q}{\R}{\S}{\notT} $  &$(a_0 \wedge b_0) \wedge (a_1 \vee b_1)$\\  
 \hline
 $-Q-R+T$ & $ \state{\notP}{\Q}{\R}{\notS}{\T} $  & $(a_0 \wedge b_0) \wedge (a_1 \vee b_1)$\\  
 \hline
 $P+S$ & $ \state{\P}{\notQ}{\notR}{\S}{\notT} $ &$(b_0 \wedge c_0) \wedge(b_1 \vee c_1)$\\
 \hline
 $P+T$ & $\state{\P}{\notQ}{\notR}{\notS}{\T}$  &$(a_0\wedge d_0) \wedge ( a_1 \vee d_1)$ \\ 
 \hline
 $-Q-R$ & $  \state{\notP}{\Q}{\R}{\notS}{\notT} $  & $(a_0\wedge b_0) \wedge ( a_1 \vee b_1)$\\ 
 \hline
\end{tabular}
}
\caption{Forbidden $q$-multiplicity formulas for Lemma \ref{contradictions to 21 possible cases}.}
\label{PQRST simple contradictions listed}
\end{table}

\begin{proof} 
Our work in the previous example has already established that $P-Q+T$ is a forbidden $q$-multiplicity formula. Next, consider the case where  $m_q(\lambda,\mu)=P+S+T$. As a consequence of \eqref{PQRSTtruelogicnotation} and \eqref{PQRSTfalselogicnotation}, the following statement must hold true:
\begin{align*}
    & 
   (a_0 \wedge b_0)  \wedge (c_0 \wedge e_0) \wedge (f_0 \wedge d_0) \wedge (c_1 \vee b_1) \wedge (a_1 \vee d_1).
\end{align*}
However, this also implies that $(a_0 \wedge d_0) \wedge (a_1 \vee d_1)$, which is a contradiction. Therefore, $P+S+T$ is a forbidden $q$-multiplicity formula. 

In Table \ref{PQRST simple contradictions listed}, we give a total of eleven cases (including the two considered above) which give rise to forbidden $q$-multiplicity formulas. Note that for each case, we specify both the necessary  condition that must be true in order for that formula to hold, as well as the contradiction that arises from such a case. 
\end{proof}

Our next result establishes $13$ additional forbidden $q$-multiplicity formulas.

\begin{lemma}\label{8 non trivial combos given by inequalites}
Let $\lambda=m\w_1+n\w_2$ and $\mu=x\w_1+y\w_2$ with $m,n,x,y\in\mathbb{N}$. If $P,Q,R,S,T$ are as in \eqref{KWMF Relabeling}, then the formulas $\sum_{x\in X}\sgn(x)x$, with $X\subseteq\{P,Q,R,S,T\}$, listed in Table~\ref{forbidden 13} 
are forbidden $q$-multiplicity formulas.
 \end{lemma}
 \begin{table}[htb!]
{
\begin{tabular}{|l|l | l |}
\hline
\textit{Case} &\textit{$m_q(\lambda,\mu)$} & \textit{Necessary  Conditions}\\
\hline
1&$P-R+T$ & $ \state{\P}{\notQ}{\R}{\notS}{\T}$ 
 \\  
\hline
2&{$P-Q+S$} & $\state{\P}{\Q}{\notR}{\S}{\notT} $ \\  
\hline

3&$-Q+S+T$ & $ \state{\notP}{\Q}{\notR}{\S}{\T}$\\  
\hline
4&$-R+S+T$ & $ \state{\notP}{\notQ}{\R}{\S}{\T}$\\  
\hline
5&$-Q+S$ & $ \state{\notP}{\Q}{\notR}{\S}{\notT}$ \\  

\hline
6&$-Q+T$ & $\state{\notP}{\Q}{\notR}{\notS}{\T}$\\ 
\hline
7&$-R+S$ & $  \state{\notP}{\notQ}{\R}{\S}{\notT}$ \\ 
\hline
8&$-R+T$ & $ \state{\notP}{\notQ}{\R}{\notS}{\T}$\\  
\hline
9&$S+T$ & $ \state{\notP}{\notQ}{\notR}{\S}{\T}$\\  
\hline 
10&$-Q$ &  $ \state{\notP}{\Q}{\notR}{\notS}{\notT} $\\ 
\hline
11&$-R$ &  $\state{\notP}{\notQ}{\R}{\notS}{\notT} $\\  
\hline
12&$S$ & $ \state{\notP}{\notQ}{\notR}{\S}{\notT}$ \\  
\hline
13&$T$ &  $ \state{\notP}{\notQ}{\notR}{\notS}{\T} $\\  
\hline
\end{tabular}
}
\caption{Forbidden $q$-multiplicity formulas for Lemma \ref{8 non trivial combos given by inequalites}.}
\label{forbidden 13}
\end{table}

\begin{proof}
We begin by describing a set of statements that give rise to contradictions. These cases will allow us to establish that the $q$-multiplicities listed in Table \ref{forbidden 13} are forbidden. \begin{enumerate}
\item[Case A:]     
Assume the statement $e_0 \wedge d_1$ holds true. If $d=m+n-x-2y-1<0$, then  $m+n-2y-1<x$. Also, if $e=n-x-2y-2\geq 0$, then $n-2y-2\geq x$. Hence, $n-2y-2>m+n-2y-1.$ Solving for $m$ explicitly yields $m<-1$, implying that whenever $e_0\wedge d_1$ holds true  the corresponding system of inequalities does not have a nonnegative integer solution.

\item[Case B:] Assume the statement $f_0 \wedge c_1$ holds true. If $c=m+3n-2x-3y-1<0$, then $m+3n-3y-1<2x$. Also, if $f=m-2x-3y-4\geq 0$, then $m-3y-4\geq 2x$. Hence, $m-3y-4>m+3n-3y-1$. Solving for $n$ explicitly yields $n<-1$, implying that this corresponding system of inequalities does not have a nonnegative integer solution. 

\item[Case C:] Assume the statement $c_0 \wedge a_1$ holds true. We observe that if $a=2m+3n-2x-3y<0,$ then $2m+3n-3y<2x.$ Also, if $c=m+3n-2x-3y-1\geq 0,$ then $m+3n-3y-1\geq 2x.$ We join these two inequalities to obtain $m+3n-3y-1>2m+3n-3y.$ If we solve for $m$ explicitly, we obtain that $m<-1,$ implying that such a system has no solutions. 

\item[Case D:] Assume the statement $d_0 \wedge b_1$ holds true. We observe that if $b=m+2n-x-2y<0,$ then $m+2n-2y<x.$ Also, if $d=m+n-x-2y-1\geq 0,$ then $m+n-2y-1\geq x.$ We join these two inequalities to obtain $m+n-2y-1>m+2n-2y.$ If we solve for $n$ explicitly, we obtain that $n<-1,$ implying that such a system has no solutions.
\end{enumerate}

Utilizing the cases above, we are now ready to consider each $q$-multiplicity listed in Table \ref{forbidden 13} and show each is forbidden. 
\begin{enumerate}
    
    \item[Case 1:] The necessary condition for $m_q(\lambda,\mu)=P-R+T$ is given by 
    \[\state{\P}{\notQ}{\R}{\notS}{\T}.\]
    Since the logical statement must hold true and it contains $a_0\wedge b_0\wedge d_0\wedge f_0$, it must be that $(c_1 \vee b_1) \wedge (c_1 \vee e_1)$ reduces to $c_1$ or $c_1\wedge e_1$. Otherwise, it would contain the contradiction $b_0 \wedge b_1$. We list all the possible ways in which the necessary condition for this case can be true and describe a contradiction arising from each possibility.
    \begin{center}
    \begin{tabular}{|l|l|}\hline
        \textit{Possible Logical Conditions} & \textit{Contradiction} \\\hline
        $\big( a_0 \wedge b_0\wedge d_0\wedge f_0 \big)  \;\wedge \; c_1$&$f_0\wedge c_1$ (Case B)\\\hline
        $\big( a_0 \wedge b_0\wedge d_0\wedge f_0 \big)  \;\wedge \;(c_1\wedge e_1)$&$f_0\wedge c_1$ (Case B)\\\hline
    \end{tabular}
    \end{center}
    
    \item[Case 2:] The necessary condition for $m_q(\lambda,\mu)=P-Q+S$ is given by 
    \[ \state{\P}{\Q}{\notR}{\S}{\notT}.\]
    Since the logical statement must hold true and it contains $a_0\wedge b_0\wedge c_0\wedge e_0$, it must be that $(\notR)\wedge (\notT)$ reduces to $d_1$ or $d_1\wedge f_1$. Otherwise, it would contain the contradiction $a_0\wedge a_1$. We list all the possible ways in which the necessary condition for this case can be true and describe a contradiction arising from each possibility.
\begin{center}
    \begin{tabular}{|l|l|}\hline
        \textit{Possible Logical Conditions} & \textit{Contradiction} \\\hline
        $(a_0\wedge b_0\wedge c_0\wedge e_0)\wedge d_1$&$e_0\wedge d_1$ (Case A)\\\hline
        $(a_0\wedge b_0\wedge c_0\wedge e_0)\wedge (d_1\wedge f_1)$&$e_0\wedge d_1$ (Case A)\\\hline
    \end{tabular}
    \end{center}
    
    \item[Case 3:] The necessary condition for $m_q(\lambda,\mu)=-Q+S+T$ is given by 
    \[ \state{\notP}{\Q}{\notR}{\S}{\T}.\]
        Since the logical statement must hold true and it contains $b_0\wedge c_0\wedge d_0\wedge e_0\wedge f_0$, it must be that $(\notP)\wedge (\notR)$ reduces to $a_1$. Otherwise, it would contain the contradiction $b_0\wedge b_1$ or $d_0\wedge d_1$. Thus, the only possible way in which the necessary condition for this case can be true is if $(b_0\wedge c_0\wedge d_0\wedge e_0\wedge f_0)\wedge a_1$ is true. However, this case contains the contradiction $c_0\wedge a_1$ as seen in Case C.
    
    \item[Case 4:] The necessary condition for $m_q(\lambda,\mu)=-R+S+T$ is given by 
    \[\state{\notP}{\notQ}{\R}{\S}{\T} .\]
    Since the logical statement must hold true and it contains $a_0\wedge c_0\wedge d_0\wedge e_0\wedge f_0$, it must be that $(\notP)\wedge (\notQ)$ reduces to $b_1$. Otherwise, it would contain the contradiction $a_0\wedge a_1$ or $c_0\wedge c_1$. Thus, the only possible way in which the necessary condition for this case can be true is if $(a_0\wedge c_0\wedge d_0\wedge e_0\wedge f_0)\wedge b_1$ is true. However, this case contains the contradiction $d_0\wedge b_1$ as seen in Case D.
    
    \item[Case 5:] The necessary condition for $m_q(\lambda,\mu)=-Q+S$ is given by 
    \[\state{\notP}{\Q}{\notR}{\S}{\notT} .\]
    Since the logical statement must hold true and it contains $b_0\wedge c_0\wedge e_0$, it must be that $(a_1 \vee b_1)  \wedge (a_1 \vee d_1) \wedge (f_1 \vee d_1)$ reduces to $a_1\wedge d_1$, $a_1\wedge f_1$, or $a_1\wedge d_1\wedge f_ 1$. Otherwise, it would contain the contradiction $b_0\wedge b_1$. Thus, there are three possible ways in which the necessary condition for this case can be true. Next, we list all the possible ways in which the necessary condition for this case can be true and describe a contradiction arising from each possibility.

\begin{center}
    \begin{tabular}{|l|l|}\hline
    \textit{Possible Logical Conditions} & \textit{Contradiction} \\\hline
    $(b_0\wedge c_0\wedge e_0)\wedge (a_1\wedge d_1)$&$e_0\wedge d_1$ (Case A)\\\hline
    $(b_0\wedge c_0\wedge e_0)\wedge (a_1\wedge f_1)$&$c_0\wedge a_1$ (Case C)\\\hline
    $(b_0\wedge c_0\wedge e_0)\wedge (a_1\wedge d_1\wedge f_1)$&$c_0\wedge a_1$ (Case C)\\\hline
    \end{tabular}
    \end{center}

    \item[Case 6:] The necessary condition for $m_q(\lambda,\mu)=-Q+T$ is given by 
    \[\state{\notP}{\Q}{\notR}{\notS}{\T} .\]
    Since the logical statement must hold true and it contains $b_0\wedge c_0\wedge d_0\wedge f_0$, it must be that $(a_1 \vee b_1)  \wedge (a_1 \vee d_1)  \wedge (c_1 \vee e_1)$ reduces to $a_1\wedge e_1$. Otherwise, it would contain the contradiction $b_0\wedge b_1$, $c_0\wedge c_1$, or $d_0\wedge d_1$. Thus, the only possible way in which the necessary condition for this case can be true is if $(b_0\wedge c_0\wedge d_0\wedge f_0)\wedge (a_1 \wedge e_1)$ is true. However, this case contains the contradiction $c_0\wedge a_1$ as seen in Case C.

    \item[Case 7:] The necessary condition for $m_q(\lambda,\mu)=-R+S$ is given by 
    \[\state{\notP}{\notQ}{\R}{\S}{\notT} .\]
    Since the logical statement must hold true and it contains $a_0\wedge c_0\wedge  d_0\wedge e_0$, it must be that $(a_1 \vee b_1) \wedge (c_1 \vee b_1)  \wedge (f_1 \vee d_1)$ reduces to $b_1\wedge f_1$. Otherwise, it would contain the contradiction $a_0\wedge a_1$, $c_0\wedge c_1$ or $d_0\wedge d_1$.
    Thus, there is only one possible way in which the necessary condition for this case can be true, namely, if $a_0\wedge c_0\wedge d_0\wedge e_0\wedge b_1\wedge f_1$ is true. However, this gives rise to the contradiction  $d_0\wedge b_1$ as seen in Case D.
    
    \item[Case 8:] The necessary condition for $m_q(\lambda,\mu)=-R+T$ is given by 
    \[\state{\notP}{\notQ}{\R}{\notS}{\T} .\]
    Since the logical statement must hold true and it contains $a_0\wedge d_0\wedge f_0$, it must be that $(a_1 \vee b_1) \wedge (c_1 \vee b_1)  \wedge (c_1 \vee e_1)$ reduces to $b_1\wedge c_1$, $b_1\wedge e_1$, or $b_1\wedge c_1\wedge e_1$. Otherwise, it would contain the contradiction $a_0\wedge a_1$.
    Thus, there are three possible ways in which the necessary condition for this case can be true. We list the three possible ways  in which the necessary condition can be true and describe a contradiction arising from each possibility.
    \begin{center}
    \begin{tabular}{|l|l|}\hline
    \textit{Possible Logical Conditions} & \textit{Contradiction} \\\hline
    $(a_0\wedge d_0\wedge f_0)\wedge (b_1\wedge c_1)$&$f_0\wedge c_1$ (Case B)\\\hline
    $(a_0\wedge d_0\wedge f_0)\wedge (b_1\wedge e_1)$&$d_0\wedge b_1$ (Case D)\\\hline
    $(a_0\wedge d_0\wedge f_0)\wedge (b_1\wedge c_1\wedge e_1)$&$f_0\wedge c_1$ (Case B)\\\hline
    \end{tabular}
    \end{center}
    
    \item[Case 9:] The necessary condition for $m_q(\lambda,\mu)=S+T$ is given by 
    \[\state{\notP}{\notQ}{\notR}{\S}{\T} .\]
    Since the logical statement must hold true and it contains $c_0\wedge d_0\wedge  e_0\wedge f_0$, it must be that $(a_1 \vee b_1) \wedge (c_1 \vee b_1) \wedge (a_1 \vee d_1)$ reduces to $a_1\wedge b_1$. Otherwise, it would contain the contradiction $c_0\wedge c_1$ or $d_0\wedge d_1$.
    Thus, there is only one possible way in which the necessary condition for this case can be true, namely, if $c_0\wedge d_0\wedge  e_0\wedge f_0\wedge a_1\wedge b_1$ is true. However, this gives rise to the contradiction  $c_0\wedge a_1$ as seen in Case C.
    \item[Case 10:] The necessary condition for $m_q(\lambda,\mu)=-Q$ is given by 
    \[\state{\notP}{\Q}{\notR}{\notS}{\notT} .\]
    Since the logical statement must hold true and it contains $b_0\wedge c_0$, it must be that $(a_1 \vee b_1)  \wedge (a_1 \vee d_1) \wedge (f_1 \vee d_1) \wedge (c_1 \vee e_1)$ reduces to $a_1\wedge e_1\wedge f_1$, $a_1\wedge e_1\wedge d_1$, or $a_1\wedge e_1\wedge d_1\wedge f_1$. Otherwise, it would contain the contradiction $b_0\wedge b_1$ or $c_0\wedge c_1$.
    Thus, there are three possible ways in which the necessary condition for this case can be true. We list these possibilities and describe a contradiction arising from each possibility.
    \begin{center}
    \begin{tabular}{|l|l|}\hline
    \textit{Possible Logical Conditions} & \textit{Contradiction} \\\hline
    $(b_0\wedge c_0)\wedge (a_1\wedge e_1\wedge f_1)$&$c_0\wedge a_1$ (Case C)\\\hline
    $(b_0\wedge c_0)\wedge (a_1\wedge e_1\wedge d_1)$&$c_0\wedge a_1$ (Case C)\\\hline
    $(b_0\wedge c_0)\wedge (a_1\wedge e_1\wedge d_1\wedge f_1)$&$c_0\wedge a_1$ (Case C)\\\hline
    \end{tabular}
    \end{center}

    \item[Case 11:] The necessary condition for $m_q(\lambda,\mu)=-R$ is given by 
    \[\state{\notP}{\notQ}{\R}{\notS}{\notT} .\]
    Since the logical statement must hold true and it contains $a_0\wedge d_0$, it must be that $(a_1 \vee b_1) \wedge (c_1 \vee b_1)  \wedge (f_1 \vee d_1) \wedge (c_1 \vee e_1)$ reduces to $b_1\wedge f_1\wedge c_1$, $b_1\wedge f_1\wedge e_1$, or $b_1\wedge f_1\wedge c_1\wedge e_1$. Otherwise, it would contain the contradiction $a_0\wedge a_1$ or $d_0\wedge d_1$.
    Thus, there are only three possible ways in which the necessary condition for this case can be true. We list these possibilities and describe a contradiction arising from each possibility.
    \begin{center}
    \begin{tabular}{|l|l|}\hline
    \textit{Possible Logical Conditions} & \textit{Contradiction} \\\hline
    $(a_0\wedge d_0)\wedge (b_1\wedge f_1\wedge c_1)$&$d_0\wedge b_1$ (Case D)\\\hline
    $(a_0\wedge d_0)\wedge (b_1\wedge f_1\wedge e_1)$&$d_0\wedge b_1$ (Case D)\\\hline
    $(a_0\wedge d_0)\wedge (b_1\wedge f_1\wedge c_1\wedge e_1)$&$d_0\wedge b_1$ (Case D)\\\hline
    \end{tabular}
    \end{center}
    
    \item[Case 12:] The necessary condition for $m_q(\lambda,\mu)=S$ is given by 
    \[\state{\notP}{\notQ}{\notR}{\S}{\notT} .\]
    Since the logical statement must hold true and it contains $c_0\wedge e_0$, it must be that $(a_1 \vee b_1) \wedge (c_1 \vee b_1) \wedge (a_1 \vee d_1) \wedge (f_1 \vee d_1)$ reduces to 
    $a_1\wedge b_1\wedge d_1 \wedge f_1$, 
    $a_1\wedge b_1\wedge d_1$, 
    $a_1\wedge b_1\wedge f_1$, 
    $b_1\wedge d_1\wedge f_1$, or 
    $b_1\wedge d_1$. 
    Otherwise, it would contain the contradiction $c_0\wedge c_1$.
    Thus, there are five possible ways in which the necessary condition for this case can be true. We list these possibilities and describe a contradiction arising from each possibility.
    
    \begin{center}
    \begin{tabular}{|l|l|}\hline
    \textit{Possible Logical Conditions} & \textit{Contradiction} \\\hline
    $(c_0\wedge e_0)\wedge(a_1\wedge b_1\wedge d_1 \wedge f_1)$&$e_0\wedge d_1$ (Case A)\\\hline
    $(c_0\wedge e_0)\wedge(a_1\wedge b_1\wedge d_1)$&$e_0\wedge d_1$ (Case A)\\\hline
    $(c_0\wedge e_0)\wedge(a_1\wedge b_1\wedge f_1)$&$c_0\wedge a_1$ (Case C)\\\hline
    $(c_0\wedge e_0)\wedge(b_1\wedge d_1\wedge f_1)$&$e_0\wedge d_1$ (Case A)\\\hline
    $(c_0\wedge e_0)\wedge(b_1\wedge d_1)$&$e_0\wedge d_1$ (Case A)\\\hline
    \end{tabular}
    \end{center}
    
    \item[Case 13:] The necessary condition for $m_q(\lambda,\mu)=T$ is given by 
\[\state{\notP}{\notQ}{\notR}{\notS}{\T} .\]
    Since the logical statement must hold true and it contains $d_0\wedge f_0$, it must be that $(a_1 \vee b_1) \wedge (c_1 \vee b_1) \wedge (a_1 \vee d_1)  \wedge (c_1 \vee e_1)$ reduces to
    $a_1\wedge b_1\wedge c_1\wedge e_1$,
    $a_1\wedge b_1\wedge c_1$,
    $a_1\wedge b_1\wedge e_1$,
    $a_1\wedge  c_1\wedge e_1$, or
    $a_1\wedge  c_1$.
    Otherwise, it would contain the contradiction $d_0\wedge d_1$.
    Thus, there are five possible ways in which the necessary condition for this case can be true. We list these possibilities and describe a contradiction arising from each possibility.
    \begin{center}
    \begin{tabular}{|l|l|}\hline
    \textit{Possible Logical Conditions} & \textit{Contradiction} \\\hline
    $(d_0\wedge f_0)\wedge(a_1\wedge b_1\wedge c_1\wedge e_1)$&$f_0\wedge c_1$ (Case B)\\\hline
    $(d_0\wedge f_0)\wedge(a_1\wedge b_1\wedge c_1)$&$f_0\wedge c_1$ (Case B)\\\hline
    $(d_0\wedge f_0)\wedge(a_1\wedge b_1\wedge e_1)$&$d_0\wedge b_1$ (Case D)\\\hline
    $(d_0\wedge f_0)\wedge(a_1\wedge  c_1\wedge e_1)$&$f_0\wedge c_1$ (Case B)\\\hline
    $(d_0\wedge f_0)\wedge(a_1\wedge  c_1)$&$f_0\wedge c_1$ (Case B)\\\hline
    \end{tabular}
    \end{center}
    
\end{enumerate}
\end{proof}

With the proof of Lemma \ref{8 non trivial combos given by inequalites} concluded, we are now prepared to give the proof of our main result. 

\begin{proof}[Proof of Theorem \ref{thm:main}]
Note that after applying Lemma \ref{contradictions to 21 possible cases} and Lemma \ref{8 non trivial combos given by inequalites} it suffices to demonstrate the existence of the remaining eight cases that are listed in the statement of Theorem \ref{thm:main}. Table \ref{combinations of non trivial pqrst with examples} provides examples of these cases.

\begin{table}[htb!]

\begin{tabular}{|l | l | l | l |}
\multicolumn{4}{c}{} \\

\hline
\textit{Evaluations} & \textit{Necessary  Conditions} &$(m,n,x,y)$ & $(a,b,c,d,e,f)$ \\ 
\hline
$P-Q-R+S+T$ & $ a_0 \wedge b_0 \wedge c_0 \wedge d_0 \wedge e_0 \wedge f_0 $& $(5,6,0,0)$ & $(28,17,22,10,4,1)$\\  
\hline
$P-Q-R+{S}$ & $ a_0 \wedge b_0 \wedge c_0 \wedge d_0 \wedge e_0 \wedge f_1$ & $(0,4,0,0)$ & $(12,8,11,3,2,-4)$\\  
\hline
$P-Q-R+{T}$ &$ a_0 \wedge b_0 \wedge c_0 \wedge d_0 \wedge e_1 \wedge f_0 $ & $(5,0,0,0)$ & $(10,5,4,4,-2,1)$\\  
\hline
$P-Q-R$ & $ a_0 \wedge b_0 \wedge c_0 \wedge d_0  
\wedge e_1 \wedge f_1$ & $(5,4,0,4)$ & $(10,5,4,0,-6,-11)$\\  
\hline
$P-Q$ & $ a_0 \wedge b_0 \wedge c_0 
 \wedge d_1 \wedge e_1 \wedge f_1$ & $(0,50,51,0)$ & $(48,49,47,-2,-3,-106)$\\ 
\hline
$P-R$ & $ a_0 \wedge b_0 \wedge c_1 \wedge d_0 \wedge e_1 \wedge f_1$ & $(2,0,1,0)$ & $(2,1,-1,0,-3,-4)$\\  
\hline 
$P$ &  $ a_0 \wedge b_0  
\wedge c_1 \wedge d_1 \wedge e_1 \wedge f_1$ & $(0,0,0,0)$ & $(0,0,-1,-1,-2,-4)$\\
\hline
$0$ & $ (a_1 \vee b_1) \wedge (c_1 \vee b_1) \wedge (a_1 \vee d_1) $& $(0,0,8,0)$ & $(-16,-8,-17,-9,-10,-20)$\\
&$\wedge (f_1 \vee d_1) \wedge (c_1 \vee e_1)$&&\\
\hline
\end{tabular}
\caption{Examples establishing the existence of certain $q$-multiplicity formulas.}
\label{combinations of non trivial pqrst with examples}
\end{table}

With the existence of these evaluations established, we now show that each evaluation implies the corresponding statement given in Theorem \ref{thm:main}. 
We first establish additional statements that give rise to contradictions. Our methods are similar to those employed in the proof of Lemma \ref{8 non trivial combos given by inequalites}.

\begin{enumerate}
    \item [Case E:] Assume the statement $a_0 \wedge f_0 \wedge d_1$ holds true. We observe that if $d=m+n-x-2y-1<0,$ then $2m+2n-4y-2<2x.$ Also, if $f=m-2x-3y-4\geq 0,$ then $m-3y-4\geq 2x.$ Finally, if $a=2m+3n-2x-3y\geq 0,$ then  $2m+3n-3y\geq 2x \geq 0,$ implying that $2m+3n\geq 3y.$ We join the first two inequalities to obtain $3y>6n+3m+6.$ We then join the inequality just obtained and the third inequality to see that $-6>3n+m.$ This is impossible since $n,m$ are non-negative, so such a system has no solution.
    \item[Case F:] Assume the statement $e_0 \wedge c_1$ holds true. We observe that if $c=m+3n-2x-3y-1<0,$ then $m+3n-3y-1<2x.$ Also, if $e=2n-2x-4y-4\geq 0,$ then $2n-4y-4\geq 2x.$ We join these two inequalities to obtain $2n-4y-4>m+3n-3y-1.$ If we solve for $m$ explicitly, we obtain that $-n-y-3>m,$ implying that such a system has no solutions.
\end{enumerate}
Utilizing these cases, we consider each $q$-multiplicity listed in Theorem \ref{thm:main}.
\begin{enumerate}
    \item [Case \RN{1}:] The necessary condition for $m_q(\lambda,\mu) = P-Q-R+S+T$ is given by
   \[ \state{\P}{\Q}{\R}{\S}{\T}.\]
   This reduces to
   \[
   a_0 \wedge b_0 \wedge c_0 \wedge d_0 \wedge e_0 \wedge f_0,
   \]
   and so $m_q(\lambda,\mu) = P-Q-R+S+T$ implies $
   a_0 \wedge b_0 \wedge c_0 \wedge d_0 \wedge e_0 \wedge f_0$.
   \item [Case \RN{2}:] The necessary condition for $m_q(\lambda,\mu) = P-Q-R+S$ is given by
   \[ \state{\P}{\Q}{\R}{\S}{\notT}.\]
   Since the logical statement must hold true and it contains $a_0 \wedge b_0 \wedge c_0 \wedge d_0 \wedge e_0$, it must be that $(f_1 \vee d_1)$ reduces to $f_1$. Otherwise, it would contain the contradiction $d_0 \wedge d_1$. Thus, there is only one possible way for the necessary condition for this case to be true. Therefore, $m_q(\lambda,\mu) = P-Q-R+S$ implies $a_0 \wedge b_0 \wedge c_0 \wedge d_0 \wedge e_0 \wedge f_1$.
   \item [Case \RN{3}:] The necessary condition for $m_q(\lambda,\mu) = P-Q-R+T$ is given by
   \[ \state{\P}{\Q}{\R}{\notS}{\T}.\]
    Since the logical statement must hold true and it contains $a_0 \wedge b_0 \wedge c_0 \wedge d_0 \wedge f_0$, it must be that $(c_1 \vee e_1)$ reduces to $e_1$. Otherwise, it would contain the contradiction $c_0 \wedge c_1$. Thus, there is only one possible way for the necessary condition for this case to be true. Therefore, $m_q(\lambda,\mu) = P-Q-R+S$ implies $a_0 \wedge b_0 \wedge c_0 \wedge d_0 \wedge e_1 \wedge f_0$.
   \item [Case \RN{4}:] The necessary condition for $m_q(\lambda,\mu) = P-Q-R$ is given by
   \[ \state{\P}{\Q}{\R}{\notS}{\notT}.\]
    Since the logical statement must hold true and it contains $a_0 \wedge b_0 \wedge c_0 \wedge d_0$, it must be that $(c_1 \vee e_1) \wedge(f_1 \vee d_1)$ reduces to $e_1 \wedge f_1$. Otherwise, it would contain the contradiction $d_0 \wedge d_1$ or $c_0 \wedge c_1$. Thus, there is only one possible way for the necessary condition for this case to be true. Therefore, $m_q(\lambda,\mu) = P-Q-R+S$ implies $a_0 \wedge b_0 \wedge c_0 \wedge d_0 \wedge e_1 \wedge f_1$.
   \item [Case \RN{5}:] The necessary condition for $m_q(\lambda,\mu) = P-Q$ is given by
   \[ \state{\P}{\Q}{\notR}{\notS}{\T}.\]
   Since the logical statement must hold true and it contains $a_0\wedge b_0 \wedge c_0$, it must be that $(a_1 \vee d_1) \wedge (c_1 \vee e_1)\wedge (f_1 \vee d_1)$ reduces to $d_1 \wedge e_1$ or $d_1 \wedge e_1 \wedge f_1$. Otherwise, it would contain the contradiction $a_0 \wedge a_1$ or $c_0 \wedge c_1$. Thus, there are two possible ways in which the necessary condition for this case can be true. However, if we consider the statement $a_0\wedge b_0 \wedge c_0 \wedge d_1 \wedge e_1 \wedge f_0$, it contains the statement $a_0\wedge f_0 \wedge d_1$, a contradiction given by Case E. Therefore, $m_q(\lambda, \mu) = P-Q$ implies $a_0 \wedge b_0\wedge c_0 \wedge d_1 \wedge e_1 \wedge f_1$.
   \item[Case \RN{6}:] The necessary condition for $m_q(\lambda,\mu) = P-R$ is given by
   \[ \state{\P}{\notQ}{\R}{\notS}{\notT}.\]
   Since the logical statement must hold true and it contains $a_0\wedge b_0 \wedge d_0$, it must be that $(c_1 \vee b_1) \wedge (c_1 \vee e_1) \wedge (f_1 \vee d_1)$ reduces to $c_1 \wedge f_1$ or $c_1 \wedge e_1 \wedge f_1$. Otherwise, it would contain the contradiction $b_0 \wedge b_1$  or $d_0 \wedge d_1$. Thus, there are two possible ways in which the necessary condition for this case can be true. However, if we consider the statement $a_0 \wedge b_0 \wedge c_1 \wedge d_0 \wedge e_0 \wedge f_1$, it contains the statement $e_0 \wedge c_1$, a contradiction given by Case F. Therefore, $m_q(\lambda,\mu) = P-R$ implies $a_0 \wedge b_0 \wedge c_1 \wedge d_0 \wedge e_1 \wedge f_1$.
   \item [Case \RN{7}:] The necessary condition for $m_q(\lambda,\mu) = P$ is given by
   \[ \state{\P}{\notQ}{\notR}{\notS}{\notT}.\]
   Since the logical statement must hold true and it contains $a_0\wedge b_0$, it must be that
   $(c_1 \vee b_1) \wedge (a_1 \vee d_1) \wedge (c_1 \vee e_1) \wedge (f_1 \vee d_1)$ reduces to $c_1 \wedge d_1$, $c_1\wedge d_1 \wedge e_1$, $c_1 \wedge d_1 \wedge f_1$, or $c_1 \wedge d_1 \wedge e_1 \wedge f_1$. Otherwise, it would contain the contradiction $a_0 \wedge a_1$ or $b_0\wedge b_1$. Thus, there are four possible ways in which the necessary condition for this case can be true. We list three of these possibilities and describe a contradiction arising from each possibility.
   \begin{center}
    \begin{tabular}{|l|l|}\hline
        \textit{Possible Logical Conditions} & \textit{Contradiction} \\\hline
        $(a_0\wedge b_0\wedge e_0\wedge f_0)\wedge (c_1 \wedge d_1)$&$a_0\wedge d_1 \wedge f_0$ (Case E)\\\hline
        $(a_0\wedge b_0\wedge  f_0)\wedge (c_1 \wedge d_1\wedge e_1)$&$a_0\wedge d_1 \wedge f_0$ (Case E)\\\hline
        $(a_0\wedge b_0\wedge  e_0)\wedge (c_1 \wedge d_1\wedge f_1)$&$e_0\wedge c_1$ (Case F)\\\hline
    \end{tabular}
    \end{center}
    Therefore, $m_q(\lambda,\mu) = P$ implies $a_0 \wedge b_0  
\wedge c_1 \wedge d_1 \wedge e_1 \wedge f_1$.
\item [Case \RN{8}:] Thus, we are left with the final case in which $m_q(\lambda,\mu)=0$.\qedhere
\end{enumerate}
\end{proof}

We now present some examples of computing weight $q$-multiplicities using our formulas.

\begin{example}\label{ex:1} If $\lambda$ is the highest root of $\mathfrak{g}_2$, i.e. $\lambda=3\a_1+2\a_2=\w_2$, and $\mu=0$, then by Theorem~\ref{thm:main} we have that $m=x=y=0$ and $n=1$ and, hence, 
$a=3$, $b=c=2$, $d=0$, $e=-1$, and $f=-4$. This implies that 
\[m_q(\lambda,\mu)=\wp_q(3\a_1+2\a_2)-\wp_q(2\a_1+2\a_2)-\wp_q(3\a_1).\]
By Proposition \ref{prop:G2qpartition} we note that 
\[
    \wp_q(3\a_1+2\a_2)=q (1 + 2 q + 2 q^2 + q^3 + q^4),\qquad
    \wp_q(2\a_1+2\a_2)=q^2 (2 + q + q^2),x\qquad\mbox{and}\qquad
    \wp_q(3\a_1)=q^3.
\]
Therefore
$m_q(\lambda,\mu)=q + q^5$,
which recovers a known result of Lusztig which shows that $m_q(\lambda,0)=\sum_{i=1}^{r}q^{e_i}$, where $\lambda$ is the highest root and $e_1,\ldots,e_r$ are the exponents of the corresponding simple Lie algebra of rank $r$ \cite{Lusztig}. In addition, note that $m(\lambda,\mu)=2$.
\end{example}

\begin{example}\label{ex:2} If $\lambda=3\w_2$ and $\mu=\w_1+2\w_2$, then by Theorem~\ref{thm:main} we have that $m=0$, $n=3$, $x=1$, $y=2$ and, hence, 
$a=1$, $b=1$, $c=0$, $d=-3$, $e=-4$, and $f=-12$. This implies that 
$m_q(\lambda,\mu)=\wp_q(\a_1+\a_2)-\wp_q(\a_2)$. By Proposition \ref{prop:G2qpartition} we note that \[\wp_q(\a_1+\a_2)=q (1 + q)\qquad\mbox{and}\qquad\wp_q(\a_2)=q.\]
Therefore $m_q(\lambda,\mu)=q^2$ and $m(\lambda,\mu)=1$. This recovers a special case of \cite[Theorem 6]{HRSS2019}.
\end{example}

We recall the following formulas for the value of Kostant's partition function for the exceptional Lie algebra $\mathfrak{g_2}$ given by Tarski. 

\begin{lemma}[Tarski p. 9-10 \cite{tarski}]\label{lem:tarski}
Let $m,n\in\mathbb{N}$.
\begin{enumerate}
    \item If $m\leq n$, then
    $\wp(m\a_1+n\a_2)=g(m)$
    \item If $n\leq m\leq \frac{3}{2}n$, then
    $\wp(m\a_1+n\a_2)=g(m)-h(m-n-1)$
    \item If $\frac{3}{2}n\leq m\leq 2n$, then $\wp(m\a_1+n\a_2)=h(n)-g(3n-m-1)+h(2n-m-2)$
    \item If $2n\leq m\leq 3n$, then $\wp(m\a_1+n\a_2)=h(m)-g(3n-m-1)$
    \item If $3n\leq m$, then $\wp(m\a_1+n\a_2)=h(n)$
\end{enumerate}
where for $k\geq -2$,
\begin{align}
    g(k)&=\begin{cases}
    \frac{1}{432}(k+6)(k^3+14k^2+54k+72)&\mbox{for $k\equiv 0\mod 6$}\\
    \frac{1}{432}(k+5)^2(k^2+10k+13)&\mbox{for $k\equiv 1\mod 6$}\\
    \frac{1}{432}(k+4)(k^3+16k^2+74k+68)&\mbox{for $k\equiv 2\mod 6$}\\
    \frac{1}{432}(k+3)^2(k+5)(k+9)&\mbox{for $k\equiv 3\mod 6$}\\
    \frac{1}{432}(k+2)(k+8)(k^2+10k+22)&\mbox{for $k\equiv 4\mod 6$}\\
    \frac{1}{432}(k+1)(k+5)(k+7)^2&\mbox{for $k\equiv 5\mod 6$}\\
    \end{cases}
\end{align}
and 
\begin{align}
h(k)&=\begin{cases}
\frac{1}{48}(k+2)(k+4)(k^2+6k+6)&\mbox{for $k$ even}\\
\frac{1}{48}(k+1)(k+3)^2(k+5)&\mbox{for $k$ odd.}\\
\end{cases}
\end{align}
\end{lemma}

We remark that one could instead use Lemma \ref{lem:tarski} along with Theorem \ref{thm:main} to compute weight multiplicities rather than setting $q=1$ in Proposition \ref{prop:G2qpartition} as we did in the above examples. We provide the details of these computations using our previous examples.

\begin{example}Following Example \ref{ex:1},
we let $\lambda=3\a_1+2\a_2=\w_2$, $\mu=0$, and by  Theorem~\ref{thm:main} we know
$m(\lambda,\mu)=\wp(3\a_1+2\a_2)-\wp(2\a_1+2\a_2)-\wp(3\a_1)$.
Using Lemma \ref{lem:tarski} parts (b), (a), and (e), respectively, we note that 
\begin{align*}
    \wp(3\a_1+2\a_2)&=g(3)-h(3-2-1)=g(3)-h(0)=\frac{1}{432}(6)^2(8)(12)-\frac{1}{48}(2)(4)(6)=7,\\
    \wp_q(2\a_1+2\a_2)&=g(2)=\frac{1}{432}(6)(2^3+16(2)^2+74(2)+68)=4,\mbox{ and }\\
    \wp_q(3\a_1)&=h(0)=1.
\end{align*}
Therefore,
$m(\lambda,\mu)=7-4-1=2$, as previously computed. 
\end{example}

\begin{example}Following Example \ref{ex:2},
we let $\lambda=3\w_2$, $\mu=\w_1+2\w_2$, and by  Theorem~\ref{thm:main} we know
$m(\lambda,\mu)=\wp(\a_1+\a_2)-\wp(\a_2)$.
Using Lemma \ref{lem:tarski} part (a) we note that 
\[
    \wp(\a_1+\a_2)=g(1)=\frac{1}{432}(6)^2(1^2+10(1)+13)=2\qquad\mbox{ and }\qquad
    \wp_q(\a_2)=g(0)=\frac{1}{432}(6)(72)=1.
\]
Therefore,
$m(\lambda,\mu)=2-1=1$, as previously computed.
\end{example}

\section{Revision of the \texorpdfstring{$q$}{q}-analog of Kostant's weight multiplicity for \texorpdfstring{$\mathfrak{sp}_4(\mathbb{C})$}{}}

\label{sec:correction}

Harris and Lauber considered the Lie algebra $\mathfrak{sp}_4(\mathbb{C})$ and gave a closed formula for the $q$-multiplicity formula. However, their partition function formula omitted an edge case, which resulted in a missing case in their work. The formula for $\wp_q(m \alpha_1 + n \alpha_2)$ given in \cite[Proposition 1.2]{harris2017weight} is correct, and we restate it here
\[\wp_q(m\alpha_1+n\alpha_2) = \sum_{i=0}^{\min(\left\lfloor\frac{m}{2}\right\rfloor, n)}\left(\sum_{j=\max(m-i,n)}^{m+n-2i}q^j\right),\]
where $m$ and $n$ are integers, $\alpha_1$ and $\alpha_2$ are the simple roots, and $\Phi^+=\{\a_1,\a_2,\a_1+\a_2,2\a_1+\a_2\}$ are the positive roots of the Lie algebra $\sp_4(\mathbb{C})$. The mistake occurs in Corollary 3.3 of \cite{harris2017weight}. We provide the corrected statement and its proof below. 

\begin{corollary}[Corrected Corollary 3.3 \cite{harris2017weight}] If $\mathfrak{g} = \mathfrak{sp}_4 (\mathbb{C})$ and  $m, n \in \mathbb{N}$, then
\setlength\extrarowheight{3pt}
\begin{align*}
\wp(m\alpha_1+n\alpha_2) &= \begin{cases}
\left(\left\lfloor\frac{m}{2}\right\rfloor + 1\right)\left(m - \left\lfloor\frac{m}{2}\right\rfloor + 1\right)&\mbox{if $n \geq m$}\\
\frac{2mn-m^2-n^2+m+n}{2} + \left\lfloor\frac{m}{2}\right\rfloor\left(m- \left\lfloor\frac{m}{2}\right\rfloor\right) + 1&\mbox{if $2n-1 > m > n$}\\
(\left\lfloor\frac{m}{2}\right\rfloor+1)(n-\frac{1}{2}\left\lfloor\frac{m}{2}\right\rfloor+1)&\mbox{if $2n>m\geq2n-1>n$}\\
\frac{(n+1)(n+2)}{2}&\mbox{if $m \geq 2n$}.\\
\end{cases}
\end{align*}
\setlength\extrarowheight{0pt}
\end{corollary}

\begin{proof}
Setting $q=1$ into equation (4) we find that
\begin{align}\label{eq:cases}
\wp(m\a_1+n\a_2)&=\left(\sum_{i=0}^{\min(\left\lfloor\frac{m}{2}\right\rfloor, n)}\min(m-i,n)\right)
-\frac{1}{2}\min\left(\left\lfloor\frac{m}{2}\right\rfloor,n\right)\left(\min\left(\left\lfloor\frac{m}{2}\right\rfloor,n\right)+1\right) \nonumber \\
&\hspace{1in}+\min\left(\left\lfloor\frac{m}{2}\right\rfloor,n\right)+1.
\end{align}
We now consider each case individually. If $n \geq m$, then equation \eqref{eq:cases} simplifies to \[ \left(\left\lfloor\frac{m}{2}\right\rfloor+1\right)\left(m-\left\lfloor\frac{m}{2}\right\rfloor+1\right).\]
If $m \geq 2n$, then equation \eqref{eq:cases} simplifies to $\frac{(n+1)(n+2)}{2}.$ If $2n-1 > m > n$, then equation \eqref{eq:cases} yields
\begin{align}
\left(\sum_{i=0}^{\left\lfloor\frac{m}{2}\right\rfloor}\min(m-i,n)\right)
-\frac{\left\lfloor\frac{m}{2}\right\rfloor\left(\left\lfloor\frac{m}{2}\right\rfloor+1\right)}{2}
+\left\lfloor\frac{m}{2}\right\rfloor+1.\label{eq:firstterm}
\end{align}
Let us consider the first term of expression~\eqref{eq:firstterm}. Since $m<2n-1$ implies that $1<m-2m+2n,$ we have that $0<\left\lfloor\frac{m-2m+2n}{2}\right\rfloor.$ Since $\left\lfloor\frac{m-2m+2n}{2}\right\rfloor=\left\lfloor\frac{m}{2}\right\rfloor-m+n,$ we see that $m-n<\left\lfloor\frac{m}{2}\right\rfloor.$ Then, because we have $m>n$, and $0<m-n$ holds. We then have that $0<m-n<\left\lfloor\frac{m}{2}\right\rfloor.$ It follows that if $i\leq m-n,$ then $n\leq m-i$ and hence $\min(m-i,n)=n.$ If $i > m-n$, then $n > m-i$ and hence $\min(m-i, n) = m-i$. Thus,
\begin{align}
\sum_{i=0}^{\left\lfloor\frac{m}{2}\right\rfloor}\min(m-i,n)
&
= \frac{2mn-m^2-n^2+m+n}{2} + m\left\lfloor\frac{m}{2}\right\rfloor 
-\frac{\left\lfloor\frac{m}{2}\right\rfloor\left(\left\lfloor\frac{m}{2}\right\rfloor+1\right)}{2}.
\label{eq:last}
\end{align}
Substituting equation \eqref{eq:last} into equation \eqref{eq:firstterm} yields the desired result. If $2n>m\geq2n-1>n,$ then $2n-m\leq 1$ which implies that $\left\lfloor\frac{m-2m+2n}{2}\right\rfloor=\left\lfloor\frac{m}{2}\right\rfloor-m+n\leq 0$, so $\left\lfloor\frac{m}{2}\right\rfloor\leq m-n.$ Thus, for all $i$ it holds that $m-n\geq i,$ implying that $m-i\geq n$ and we obtain that
\begin{align}
\sum_{i=0}^{\left\lfloor\frac{m}{2}\right\rfloor}\min(m-i,n)
&
= n \left(\left\lfloor\frac{m}{2}\right\rfloor+1\right).
\label{eq:lastlast}
\end{align}
Substituting equation \eqref{eq:lastlast} into equation \eqref{eq:firstterm} yields the desired result.
\end{proof}

As a consequence of this correction to Corollary 3.3 of \cite{harris2017weight}, the following result replaces Corollary 4.1 in \cite{harris2017weight}.
\noindent \begin{corollary}[Corrected Corollary 4.1 \cite{harris2017weight}]
Let $\lambda = m\w_1 + n\w_2$ and $\mu = x\w_1 + y\w_2$ with $m,n,x,y \in \mathbb{N}:=\{0,1,2,\ldots\}$ be weights of $\sp_4(\mathbb{C})$ and define 
$a=m+n-x-y$,
$b= n-y+\frac{m-x}{2}$,
$c= n-x-y-1$, and
$d= -y-1+\frac{m-x}{2}.$   
Then
\begin{equation}
\label{mainq1}
m(\lambda, \mu) = \begin{cases}
P - Q - R&\mbox{if $a,b,c,d\in\mathbb{N}$}\\
P-Q&\mbox{if $a,b,c\in\mathbb{N}$ and $d\notin\mathbb{N}$}\\
P-R&\mbox{if $a,b,d\in\mathbb{N}$ and $c\notin\mathbb{N}$}\\
P&\mbox{if $a,b\in\mathbb{N}$ and $c,d\notin\mathbb{N}$}\\
0&\mbox{otherwise}
\end{cases}
\end{equation}
where
\begin{align*}
P &= \begin{cases}
\left(\left\lfloor\frac{a}{2}\right\rfloor + 1\right)\left(a - \left\lfloor\frac{a}{2}\right\rfloor + 1\right)&\mbox{if $b \geq a$}\\
\frac{(b+1)(b+2)}{2}&\mbox{if $a \geq 2b$}\\
\frac{2ab-a^2-b^2+a+b}{2} + \left\lfloor\frac{a}{2}\right\rfloor\left(a- \left\lfloor\frac{a}{2}\right\rfloor\right) + 1&\mbox{if $2b-1 > a > b$}\\
(\left\lfloor\frac{a}{2}\right\rfloor+1)(b-\frac{1}{2}\left\lfloor\frac{a}{2}\right\rfloor+1)&\mbox{if $2b>a\geq2b-1>b$,}\\
\end{cases}\\ 
Q &= \left\lfloor\frac{c+2}{2}\right\rfloor^2,\\
R &=\frac{(d+1)(d+2)}{2}.
\end{align*}
\end{corollary}

\section{Future work}
Finding formulas for  Kostant's partition has recently been connected to counting multiplex juggling sequences \cite{juggling,HIO}. These bijections have been considered for all classical Lie algebras, but extending them to the exceptional Lie algebras, such as $\mathfrak{g}_2$, remains an open problem. 
For a second direction of research, we remark that one could consider giving explicit formulas for the $q$-analog of Kostant's partition function for $\mathfrak{g}_2$. This would require working through the expansion of Proposition \ref{prop:G2qpartition} using the coefficient constraints given by Tarski in Lemma \ref{lem:tarski}. We omitted such a computation because of its tedious and technical nature.

\section*{Acknowledgements}
This research was supported in part by the Alfred P. Sloan Foundation, the Mathematical Sciences Research Institute, and the National Science Foundation. We thank Rebecca Garcia for her many helpful conversations.
\bibliographystyle{plain}
\bibliography{references} 
\end{document}